%% file: paper.tex
\documentclass{amsart}

\input{macro}

\title[Tropical Rational Functions of Degree 3]{The Tropical Moduli Space of Degree-3 Rational Functions}

\author{Tony Shaska \and Mohammad-Reza Siadat}
\address{Department of Computer Science and Engineering,\\
School of Engineering and Computer Science\\
Oakland University, Rochester, MI, 48309}
\email{shaska@oakland.edu}
\email{siadat@oakland.edu}

\date{}
\hypersetup{
  pdftitle={The Tropical Moduli Space of Degree-3 Rational Functions},
  pdfauthor={Tony Shaska and Mohammad-Reza Siadat}
}

\begin{document}

\begin{abstract}
We study tropical rational functions on the tropical projective line with three zeros and three poles in the torus, counted with multiplicity, modulo target automorphisms. The principal divisor identifies their moduli space with unordered pairs of disjoint effective divisors of degree three. We classify its \(53\) cells and prove that it is six-dimensional with ten contractible connected components. Source--target automorphism groups are trivial or of order two, and reflection loci are defined by linear gap equations. Allowing common support and points at infinity gives a divisor-pair compactification. Over an algebraically closed non-Archimedean field with value group \(\R\), root--pole tropicalization is surjective from valuation-separated cubic rational functions modulo the target torus normalizer. Explicit counterexamples show that the valuation of the conjugacy invariant \(J^2/I^3\) does not descend to this tropical moduli space. Generic functions admit minimal six-unit shallow ReLU representations with three positive and three negative unit output weights.
\end{abstract}

\maketitle

\section{Introduction}\label{sec-1}

Rational functions on the projective line connect algebraic dynamics, Hurwitz theory, and invariant theory. In arithmetic dynamics the usual equivalence is conjugation by \(\PGL_2\); the cubic conjugacy space \(\cM_3^1\) has dimension four and an explicit invariant-theoretic description \cite{2024-04}. We study a different moduli problem: scalar tropical rational functions modulo integral-affine automorphisms of the target. Source translations and reflections act on our space but are not part of its defining equivalence.

On \(\TP^1=\R\cup\{-\infty,+\infty\}\), a tropical rational function is integral piecewise-linear on \(\R\), and its principal divisor determines it up to an additive constant. We define its divisor degree as the common degree of the positive and negative parts of that divisor. Requiring degree three and finite divisor support gives three zeros and three poles, counted with multiplicity, and zero asymptotic slopes. The correction to the degree convention of the first version is recorded in \cref{rem:earlier-normalization}.

The resulting space has the global description
\[
\cM_3^{\trop}
\cong
\Bigl\{(D_+,D_-)\in\Sym^3(\R)\times\Sym^3(\R):
\supp(D_+)\cap\supp(D_-)=\varnothing\Bigr\}/\mathfrak S_2.
\]
Here \(\Sym^3(\R)\) parametrizes effective degree-three divisors, and \(\mathfrak S_2\) exchanges the factors. The identification follows from the principal divisor: target translations remove the additive constant, while target reflection exchanges zeros and poles (\cref{thm:global-model}).

The cells are indexed by nonzero integral slope-jump sequences with positive and negative masses both three. We enumerate \(53\) target types, or \(31\) types after additionally identifying source reversal (\cref{thm:degree3-enumeration}). Each of the ten generic cells has six distinct divisor points. Same-sign collisions supply its lower-dimensional faces; opposite-sign collisions leave the exact-degree-three locus. Consequently the space has dimension six and ten contractible connected components (\cref{prop:components}). We also determine all affine source--target automorphisms: the groups are trivial or of order two, with reflection loci specified by anti-palindromic jumps and symmetric gaps (\cref{thm:symmetry-criterion}). The divisor construction and component description extend directly to arbitrary positive degree (\cref{rem:general-degree}).

Allowing the two divisors to meet or acquire support at the ends gives a compactification by the quotient of \(\Sym^3(\TP^1)\times\Sym^3(\TP^1)\) under factor exchange. It remembers common divisor mass rather than cancelling it. A generic zero--pole collision is therefore recorded by a degree-two function together with its cancellation point (\cref{prop:degree2-boundary}).

Root--pole tropicalization provides the algebraic comparison. Over an algebraically closed non-Archimedean field with value group \(\R\), every tropical class is realized by a cubic whose zero valuations are disjoint from its pole valuations, modulo the target torus normalizer (\cref{thm:surjective-trop}). This is not a tropicalization theorem for the full \(\PGL_2\)-conjugacy quotient. Explicit counterexamples show that the valuation of \(J^2/I^3\) is not determined by the root--pole tropicalization and can change under target scaling even on a generic tropical cell (\cref{prop:invariant-nondescent}).

The hinge representation also places the space within the tropical-rational/ReLU correspondence \cite{ZhangNaitzatLim18}. Generic functions have minimal shallow representations with six distinct thresholds and three unit output weights of each sign, with parameter equivalence governed by \cite{DereichKassing22}. Finally, scalar tropical functions need not satisfy harmonic balancing; \cref{sec-7} compares their divisor geometry with tropical Hurwitz theory without identifying the two moduli problems.

\section{Background on Tropical Rational Functions}\label{sec-2}

\subsection{The tropical semiring}
The foundational object is the tropical semiring \(\mathbb{T} = (\mathbb{R} \cup \{-\infty\}, \oplus, \odot)\), defined by
\[
a \oplus b = \max(a,b), \qquad a \odot b = a + b,
\]
where \(-\infty\) acts as the additive identity. Tropical powers are written \(a^{\odot k} = ka\). Although \(\mathbb{T}\) has no additive inverses, we will use classical subtraction \(a-b\) when convenient; this corresponds to division in the multiplicative structure \cite{MaclaganSturmfels15}*{Section~1.1}. We use the max-plus convention throughout; \cite{MaclaganSturmfels15} uses min-plus, which reverses the sign of valuations (see \cref{sec:coefficientwise}).

A \emph{tropical polynomial} of degree \(d\) in one variable is a formal expression
\[
f(x) = \bigoplus_{i=0}^d a_i \odot x^{\odot i} = \max_{0\leq i\leq d}(a_i + i x),
\]
with \(a_i \in \mathbb{T}\) and \(a_d \neq -\infty\). As a real-valued function \(f\colon \mathbb{R}\to\mathbb{R}\), every tropical polynomial is continuous and piecewise-linear with integer slopes. The points where the maximizing term changes are called \emph{break points} (or \emph{kinks}). At a break point \(x_0\) the slope increases by the difference between the largest and the smallest indices of the terms that attain the maximum; this difference is the multiplicity of \(x_0\) as a tropical root \cite{MaclaganSturmfels15}*{Section~1.1}. We use the term tropical polynomial both for the formal expression and for the function it defines. Distinct expressions may define the same function, and statements about uniqueness below always refer to functions.

\subsection{Rational functions and principal divisors}
A \emph{tropical rational function} on \(\TP^1\) is a continuous function
\[
F:\TP^1\longrightarrow\R\cup\{\pm\infty\}
\]
whose restriction to \(\R\) is piecewise-linear with integral slopes and finitely many pieces. The values \(\pm\infty\) may occur only at the two ends.

For \(p\in\TP^1\), the order \(\ord_p(F)\) is the sum of the outgoing slopes of \(F\) at \(p\). The principal divisor of \(F\) is
\[
\divisor(F)=\sum_{p\in\TP^1}\ord_p(F)\,p.
\]
Every principal divisor has degree zero. If \(\theta_1<\cdots<\theta_k\) are the finite break points and \(s_i\) is the slope on the interval between \(\theta_i\) and \(\theta_{i+1}\), then
\[
\ord_{\theta_i}(F)=s_i-s_{i-1}=: \delta_i.
\]
At the ends,
\[
\ord_{-\infty}(F)=s_-,
\qquad
\ord_{+\infty}(F)=-s_+,
\]
where \(s_-\) and \(s_+\) are the left and right asymptotic slopes.

Write
\[
\divisor(F)=D_+(F)-D_-(F)
\]
with \(D_+(F),D_-(F)\) effective and disjointly supported after cancellation.

\begin{defn}[Divisor degree]\label{def:degree}
For a tropical rational function \(F\), its \emph{divisor degree} is
\[
\deg_{\mathrm{div}}(F)
:=\deg D_+(F)=\deg D_-(F).
\]
We call \(F\) a degree-$d$ tropical rational function if \(\deg_{\mathrm{div}}(F)=d\). In particular, constant functions have divisor degree zero.
\end{defn}

The equality of the two degrees follows because \(\deg\divisor(F)=0\). This is the degree convention used throughout the paper.

\begin{rem}[Earlier degree convention]\label{rem:earlier-normalization}
The normalization of arXiv:2605.15347v1 imposed asymptotic slopes \(3\) at both ends and total finite slope variation \(4\). Its positive and negative finite divisor masses are therefore both \(2\), while the endpoint contributions \(3[-\infty]-3[+\infty]\) give divisor degree \(5\). For a function \(\varphi\) in that class, \(G(x)=\varphi(x)-3x\) extends to the ends with zero asymptotic slopes and has divisor degree \(2\). Subtracting \(3x\) is not a target automorphism and is not part of the equivalence relation used here. The present paper studies divisor degree \(3\).
\end{rem}

\begin{prop}[Functions on the tropical projective line are determined by their divisors]\label{prop:tree-divisor}
Let \(D\) be a degree-zero divisor on \(\TP^1\). Then there exists a tropical rational function \(F\), unique up to an additive constant, such that \(\divisor(F)=D\).
\end{prop}

\begin{proof}
Let \(c_-\) and \(c_+\) be the coefficients of \(D\) at \(-\infty\) and \(+\infty\), and let \(\theta_1<\cdots<\theta_k\) be its finite support points with coefficients \(d_i\). Define a step function on \(\R\) with initial value \(c_-\) and jump \(d_i\) at \(\theta_i\); if \(k=0\), take the constant step function \(c_-\). Its final value is \(c_-+\sum d_i=-c_+\), since \(\deg D=0\). Integrating it gives a continuous piecewise-linear function \(F\) with integral slopes, unique up to an additive constant, and by the definition of order at finite points and at the ends, \(\divisor(F)=D\). Conversely any \(F\) with \(\divisor(F)=D\) has this slope function, so \(F\) is determined up to a constant.
\end{proof}

On the open locus where \(D_+\) and \(D_-\) are supported in \(\R\), the asymptotic slopes are zero. In particular, a generic degree-$3$ function is bounded at both ends and has six simple finite divisor points.

\subsection{Hinge coordinates and ReLU networks}
Let \(\theta_1<\cdots<\theta_k\) be the finite break points of \(F\), let \(s_-\) be the left asymptotic slope, and let \(\delta_i\) be the slope jump at \(\theta_i\). Then
\begin{equation}\label{eq:hinge}
F(x)=a+s_-x+\sum_{i=1}^k\delta_i\max\{0,x-\theta_i\}.
\end{equation}
This is the standard one-dimensional hinge representation.

The connection with neural networks is established in \cite{ZhangNaitzatLim18}*{Cor.~5.3, Thm.~5.4, and Prop.~5.5}: feedforward ReLU networks with integral weight matrices, real biases, and linear output represent precisely the tropical rational functions, which in one variable are exactly the continuous piecewise-linear functions with integral slopes. In one variable, \cref{eq:hinge} makes the correspondence completely explicit. The break points are activation thresholds, the slope jumps are output coefficients, and the affine term records the baseline slope and bias. Recent work further develops real tropical parameter spaces of neural-network functions \cite{BrandenburgLohoMontufar24}, and related tropical rational signomial descriptions occur for graph neural networks \cite{PhamGarg24}.

For the degree-$3$ open locus considered below, \(s_-=0\) and
\[
\sum_{\delta_i>0}\delta_i=3,
\qquad
-\sum_{\delta_i<0}\delta_i=3.
\]
Thus the generic degree-$3$ ReLU function has six distinct thresholds with three output coefficients \(+1\) and three output coefficients \(-1\). Lower-dimensional strata are obtained by merging thresholds of the same sign, producing output coefficients of magnitude \(2\) or \(3\). This is a natural constrained family inside the general one-dimensional shallow ReLU class. For minimality and parameter equivalence we use real-parameter shallow networks with a linear output, an output bias, and no affine skip connection, as in \cite{DereichKassing22}; the normalized hinge representatives below have integral weights. By \cite{DereichKassing22}*{arXiv version, Remark~3.3}, a non-affine univariate function with \(n\) break points has a minimal representation with \(n\) or \(n+1\) units; a nonconstant affine function needs two units \cite{DereichKassing22}*{arXiv version, Remark~3.4}. In the first case the minimal representations correspond, up to permutation and positive rescaling of the units, to the sign vectors \(\sigma\in\{\pm1\}^n\) for which the left asymptotic slope plus the sum of the jumps \(\delta_i\) with \(\sigma_i=-1\) vanishes; the unit at \(\theta_i\) is \(\operatorname{ReLU}(x-\theta_i)\) if \(\sigma_i=1\) and \(\operatorname{ReLU}(\theta_i-x)\) if \(\sigma_i=-1\). On the locus \(s_-=0\) the hinge formula \cref{eq:hinge}, which has \(\sigma=(1,\ldots,1)\), is therefore minimal. It is not the only minimal representation: for a generic degree-$3$ function the admissible \(\sigma\) are those that reverse equally many positive and negative unit jumps, so there are \(\sum_{i=0}^{3}\binom{3}{i}^2=20\) classes of minimal representations. The break-point and slope-jump coordinates used throughout are the breakpoint and delta-slope coordinates of \cite{SahsEtAl22}.

Finally, on the torus locus, where the asymptotic slopes vanish, the representation \(F=f-g\) with \(f,g\) tropical polynomial functions of degree \(\deg_{\mathrm{div}}F\) and finite roots is unique up to a common additive constant. Indeed, if \(R_f\) and \(R_g\) are the root divisors, then \(R_f-R_g=D_+(F)-D_-(F)\) by \cref{prop:tree-divisor}, and \(\deg R_f=\deg D_+(F)\) forces \(R_f=D_+(F)+E\) and \(R_g=D_-(F)+E\) with \(E\) effective of degree \(0\), so \(E=0\); a tropical polynomial function with finite roots is determined by its roots up to an additive constant. The statement concerns functions, not formal expressions: for instance \(\max\{0,c+x,3x\}=\max\{0,3x\}\) for every \(c\le0\). This is the dimension-one case of the minimal-representation problem studied by Tran and Wang \cite{TranWang24} and by Sukenaga \cite{Sukenaga24}. In dimension two, minimum-volume representations need not be unique \cite{Sukenaga24}*{Prop.~3.12}, and the monomial and factorization complexities of \cite{TranWang24} need not agree. Brandenburg, Grillo, and Hertrich \cite{BrandenburgGrilloHertrich25}*{Section~6.2.3, Example~6.9} give a three-dimensional counterexample to the local-to-global balancing construction proposed in \cite{TranWang24}*{Section~4.1}. The integer-versus-real distinction is the one drawn in \cite{ZhangNaitzatLim18}*{Prop.~5.6} and \cite{CharisopoulosMaragos18}: with real weights a shallow network computes a tropical rational signomial, and the discrete type below reduces to the ordered sign pattern of the jumps, their magnitudes becoming continuous parameters.

\subsection{Scalar tropical functions versus harmonic morphisms}
A tropical rational function in the sense above is a scalar integral PL function. Automorphism groups of metric graphs in general are studied in \cites{JoynerKsirMelles10,Song21}. A finite harmonic morphism between metric graphs is more restrictive: in addition to being integer-affine on edges, it satisfies a balancing condition at every point, and its local degree is independent of the chosen target tangent direction. These notions should not be conflated. Tropical Hurwitz spaces parametrize balanced covers, whereas the present paper studies scalar tropical rational functions through their principal divisors. We return to this distinction in \cref{sec-7}.

\section{The Tropical Projective Line and Its Automorphisms}\label{sec-3}

\subsection{The tropical projective line}
The tropical projective line is
\[
\TP^1=\R\cup\{-\infty,+\infty\},
\]
which is topologically a closed interval. Equivalently,
\[
\TP^1=(\mathbb T^2\setminus\{(-\infty,-\infty)\})/\sim,
\qquad
(a,b)\sim(a+\lambda,b+\lambda).
\]
The finite part \(\R\) is the tropical torus.

\subsection{Automorphisms}
The group of tropical, that is integral-affine, automorphisms of \(\TP^1\) consists of the affine maps
\[
x\longmapsto \varepsilon x+b,
\qquad
\varepsilon\in\{\pm1\},\ b\in\R,
\]
and is isomorphic to \(\R\rtimes\Z/2\Z\). The orientation-preserving subgroup consists of translations, while the nontrivial finite-order elements are the reflections \(x\mapsto -x+b\).

Post-composition acts on tropical rational functions. A target translation adds a constant to \(F\) and therefore does not change \(\divisor(F)\). A target reflection sends \(F\) to \(-F+b\) and exchanges the positive and negative parts of the divisor. This motivates the following moduli equivalence.

\begin{defn}\label{def:moduli-equivalence}
Two degree-$3$ tropical rational functions are \emph{target-equivalent} if they differ by post-composition with an automorphism of \(\TP^1\). Thus additive constants are forgotten and \(F\) is identified with \(-F\). Throughout, \(\cM_3^{\trop}\) denotes the set of target-equivalence classes of degree-$3$ functions whose zeros and poles lie in \(\R\), the torus locus, equipped with the quotient topology from the break-point coordinates and the induced polyhedral stratification of \cref{sec-5}; it is an open polyhedral cell complex, not a polyhedral complex, because the opposite-sign collision faces are absent; functions with zeros or poles at \(\pm\infty\) appear only in the compactification \(\overline{\cM}_3^{\trop}\) of \cref{sec-8}.
\end{defn}

Source translations and source reflections are not quotiented out in \(\cM_3^{\trop}\); they act on it. A source translation shifts every breakpoint by the same amount. A source reflection reverses the ordered breakpoint configuration and sends a jump sequence
\[
(\delta_1,\ldots,\delta_k)
\quad\text{to}\quad
(\delta_k,\ldots,\delta_1).
\]
Target reflection sends the same sequence to \((-\delta_1,\ldots,-\delta_k)\).

\subsection{Degree and the open torus locus}
We focus first on the open locus of degree-$3$ functions whose zero and pole divisors are supported in \(\R\). Then the asymptotic slopes are zero and the divisor takes the form
\[
\divisor(F)=\sum_{i=1}^k\delta_i\theta_i,
\qquad
\theta_1<\cdots<\theta_k,
\]
where every \(\delta_i\ne0\) and
\begin{equation}\label{eq:degree3mass}
\sum_{\delta_i>0}\delta_i=3,
\qquad
-\sum_{\delta_i<0}\delta_i=3.
\end{equation}
Boundary zeros or poles at \(\pm\infty\) will appear naturally in the compactification of \cref{sec-8}.

\section{The Global Model and Combinatorial Types}\label{sec-4}

\subsection{Global divisor-pair description}
Let \(\Sym^3(\R)\) denote the third symmetric power of the finite tropical line. Define
\[
U_3=
\Bigl\{(D_+,D_-)\in\Sym^3(\R)\times\Sym^3(\R):
\supp(D_+)\cap\supp(D_-)=\varnothing\Bigr\}.
\]
The involution \(\iota(D_+,D_-)=(D_-,D_+)\) corresponds to target reflection.

\begin{thm}[Global model]\label{thm:global-model}
The target-equivalence classes of divisor-degree-$3$ tropical rational functions with finite zero and pole divisors are naturally identified with
\[
\cM_3^{\trop}\cong U_3/\langle\iota\rangle.
\]
\end{thm}

\begin{proof}
By \cref{prop:tree-divisor}, the degree-zero divisor \(D_+-D_-\) determines a tropical rational function uniquely up to an additive constant, which is exactly target translation. Target reflection replaces \(F\) by \(-F+b\), hence interchanges \(D_+\) and \(D_-\). Conversely every target-equivalence class determines the unordered pair \(\{D_+,D_-\}\).
\end{proof}

This description makes the degree condition intrinsic and avoids introducing an artificial asymptotic slope.

\subsection{Jump sequences in degree \texorpdfstring{$3$}{3}}
Let \(F\) have divisor degree \(3\) and finite support. Write the break points as
\[
\theta_1<\theta_2<\cdots<\theta_k
\]
and set
\[
\delta_i=\ord_{\theta_i}(F)\in\Z\setminus\{0\}.
\]
Then \cref{eq:degree3mass} holds. Conversely, any ordered nonzero integer sequence satisfying \cref{eq:degree3mass} determines the slope data of a degree-$3$ tropical rational function. Starting from left slope zero, the successive slopes are
\[
s_i=\delta_1+\cdots+\delta_i,
\qquad
s_0=s_k=0.
\]
No positivity condition on the slopes is required: a tropical rational function need not be monotone.

\begin{defn}\label{def:jump-types}
An \emph{oriented jump type} is an ordered sequence of nonzero integers \(\delta=(\delta_1,\ldots,\delta_k)\) satisfying \cref{eq:degree3mass}. Target reflection acts by
\[
T(\delta)=(-\delta_1,\ldots,-\delta_k),
\]
and source reversal acts by
\[
R(\delta)=(\delta_k,\ldots,\delta_1).
\]
A \emph{target type} is an orbit under \(T\). A \emph{source-unoriented type} is an orbit under the group generated by \(T\) and \(R\).
\end{defn}

\subsection{Enumeration of all degree-\texorpdfstring{$3$}{3} types}
\begin{thm}[Enumeration]\label{thm:degree3-enumeration}
The degree-$3$ jump sequences are distributed as follows.
\[
\begin{array}{c|ccccc|c}
\text{number of break points }k&2&3&4&5&6&\text{total}\\ \hline
\text{oriented sequences}&2&12&32&40&20&106\\
\text{modulo target reflection}&1&6&16&20&10&53\\
\text{modulo target reflection and source reversal}&1&3&10&10&7&31
\end{array}
\]
In particular, the generic locus has six simple break points and has \(20\) oriented types, \(10\) target types, and \(7\) source-unoriented types.
\end{thm}

\begin{proof}
Suppose the sequence has \(p\) positive jumps and \(q\) negative jumps. The positive jump magnitudes form a composition of \(3\) into \(p\) positive parts, giving \(\binom{2}{p-1}\) choices, and the negative magnitudes form a composition of \(3\) into \(q\) positive parts, giving \(\binom{2}{q-1}\) choices. The signs can be interleaved in \(\binom{p+q}{p}\) ways. Therefore
\[
\sum_{p=1}^3\sum_{q=1}^3
\binom{2}{p-1}\binom{2}{q-1}\binom{p+q}{p}=106.
\]
Collecting terms by \(k=p+q\) gives \(2,12,32,40,20\).

No nonzero sequence is fixed by target reflection, so quotienting by \(T\) divides these numbers by two. For the quotient by the four-element group \(\{1,T,R,TR\}\), Burnside's lemma applies. No sequence is fixed by \(T\). Nor is any degree-$3$ sequence fixed by \(R\). Indeed, if a palindromic sequence has even length, every positive jump occurs in a pair, so its positive divisor degree is even. If it has odd length, write the central entry as \(c\). The zero-sum condition forces \(c\) to be even, and the positive degree is then twice the positive mass in one half plus \(\max(c,0)\), again even. This contradicts degree three. The anti-palindromic condition \(TR(\delta)=\delta\), namely \(\delta_i=-\delta_{k+1-i}\), has respectively \(2,0,8,0,8\) solutions for \(k=2,3,4,5,6\). Hence the orbit counts are
\[
\frac{2+2}{4}=1,\quad
\frac{12}{4}=3,\quad
\frac{32+8}{4}=10,\quad
\frac{40}{4}=10,\quad
\frac{20+8}{4}=7.
\]
\end{proof}

For completeness, representatives for the \(31\) source-unoriented classes are listed below. In each orbit we choose a representative whose first jump is positive.

\begin{center}
\small
\begin{tabular}{c|>{\raggedright\arraybackslash}p{0.82\textwidth}}
\(k\)&representatives \(\delta=(\delta_1,\ldots,\delta_k)\)\\ \hline
2 & \((3,-3)\)\\[1mm]
3 & \((2,-3,1)\), \((3,-2,-1)\), \((3,-1,-2)\)\\[1mm]
4 & \((1,-2,2,-1)\), \((1,1,-3,1)\), \((1,2,-2,-1)\), \((2,-2,-1,1)\), \((2,-2,1,-1)\), \((2,-1,-2,1)\), \((2,-1,1,-2)\), \((2,1,-2,-1)\), \((2,1,-1,-2)\), \((3,-1,-1,-1)\)\\[1mm]
5 & \((1,-1,-1,2,-1)\), \((1,-1,1,-2,1)\), \((1,1,-2,-1,1)\), \((1,1,-2,1,-1)\), \((1,1,-1,-2,1)\), \((1,2,-1,-1,-1)\), \((2,-1,-1,-1,1)\), \((2,-1,-1,1,-1)\), \((2,-1,1,-1,-1)\), \((2,1,-1,-1,-1)\)\\[1mm]
6 & \((1,-1,-1,1,1,-1)\), \((1,-1,1,-1,-1,1)\), \((1,-1,1,-1,1,-1)\), \((1,1,-1,-1,-1,1)\), \((1,1,-1,-1,1,-1)\), \((1,1,-1,1,-1,-1)\), \((1,1,1,-1,-1,-1)\)
\end{tabular}
\end{center}

\subsection{The seven generic types}
Since the generic jumps all have magnitude one, the seven source-unoriented generic representatives may be written more transparently as
\begin{align*}
&(1,1,1,-1,-1,-1),\\
&(1,1,-1,1,-1,-1),\\
&(1,1,-1,-1,1,-1),\\
&(1,1,-1,-1,-1,1),\\
&(1,-1,1,-1,1,-1),\\
&(1,-1,1,-1,-1,1),\\
&(1,-1,-1,1,1,-1).
\end{align*}
The first, second, fifth, and seventh are anti-palindromic and therefore support reflection symmetry when their metric gaps are symmetric; this is proved in \cref{sec-6}.

\begin{exa}[A generic degree-$3$ function]\label{exa:degree3}
Take zeros at \(0,1,2\) and poles at \(4,5,7\). Then
\[
\begin{aligned}
F(x)={}&\operatorname{ReLU}(x)+\operatorname{ReLU}(x-1)+\operatorname{ReLU}(x-2)\\
&-\operatorname{ReLU}(x-4)-\operatorname{ReLU}(x-5)-\operatorname{ReLU}(x-7).
\end{aligned}
\]
Its jump sequence is \((1,1,1,-1,-1,-1)\), its slope sequence is
\[
(0,1,2,3,2,1,0),
\]
and
\[
\divisor(F)=(0)+(1)+(2)-(4)-(5)-(7)
\]
in the evident divisor notation. The gap vector is \((1,1,2,1,2)\), so this particular representative is not reflection-symmetric.
\end{exa}

\section{Cell Structure and Topology}\label{sec-5}

\subsection{Cells for fixed combinatorial type}
Fix an oriented jump sequence \(\delta=(\delta_1,\ldots,\delta_k)\) and break points
\[
\theta_1<\cdots<\theta_k.
\]
Set
\[
\ell_i=\theta_{i+1}-\theta_i>0,
\qquad i=1,\ldots,k-1.
\]
The break-point configuration is equivalent to the gap vector \((\ell_1,\ldots,\ell_{k-1})\) together with one source-position parameter \(a\in\R\). Therefore the oriented cell is
\[
C_\delta^\circ\cong\R_{>0}^{k-1}\times\R.
\]
Target reflection identifies the cell for \(\delta\) with the cell for \(-\delta\), without changing the breakpoint coordinates.

\subsection{Collisions and the cell structure}
Let \(\ell_i\to0\), so that two consecutive divisor points collide.

If \(\delta_i\) and \(\delta_{i+1}\) have the same sign, they merge to a single jump
\[
\delta_i+\delta_{i+1}
\]
of the same sign. The positive and negative divisor degrees remain \(3\), so this is an internal degeneration of \(\cM_3^{\trop}\). In the divisor-pair model it is simply a collision of two zeros or two poles, producing multiplicity.

If the jumps have opposite signs, say \(a>0\) and \(-b<0\), then the combined jump is \(a-b\). After cancellation, the positive and negative divisor degrees both drop by
\begin{equation}\label{eq:degree-drop}
\min(a,b).
\end{equation}
Thus the scalar limit has lower divisor degree. Such a collision lies outside the exact degree-$3$ locus and belongs to the resultant boundary of the compactification in \cref{sec-8}.

\begin{prop}[Cell counts]\label{prop:cell-counts}
The exact degree-$3$ moduli space has, by number of break points \(k\),
\[
1,\ 6,\ 16,\ 20,\ 10
\]
target-reflection cell types for \(k=2,3,4,5,6\), respectively. The ten \(k=6\) cells are the maximal cells.
\end{prop}

\begin{proof}
This is the second row of the table in \cref{thm:degree3-enumeration}.
\end{proof}

\subsection{Dimension and source reduction}\label{sec:source-reduction}
A cell with \(k\) break points has dimension \(k\). Since \(k\le6\),
\[
\dim\cM_3^{\trop}=6.
\]
The generic locus has six distinct divisor points and is six-dimensional.

Source translation acts freely on the finite-support locus by shifting all break points. We define the reduced space
\[
\cM_{3,\mathrm{red}}^{\trop}
:=\cM_3^{\trop}/\R_{\mathrm{source}}.
\]
On a maximal cell it has coordinates \((\ell_1,\ldots,\ell_5)\in\R_{>0}^5\) and dimension \(5\). Source reflection acts by reversing the gap vector and reversing the jump sequence. Modulo this additional involution, the generic combinatorics reduces from ten target types to the seven classes listed in \cref{sec-4}.

\subsection{Connected components}
Expanding each jump \(\delta_i\) into \(|\delta_i|\) consecutive unit jumps of the same sign assigns to every degree-$3$ function with finite divisor a word \(w\) with three letters \(+1\) and three letters \(-1\): the interlacing pattern of its zeros and poles, counted with multiplicity. Let \(r(w)\) denote the number of maximal runs of equal letters in \(w\).

\begin{prop}[Connected components]\label{prop:components}
For each word \(w\), the set \(U_w\subset U_3\) of divisor pairs with word \(w\) is a connected component of \(U_3\), and
\[
U_w\cong\R\times(0,\infty)^{r(w)-1}\times[0,\infty)^{6-r(w)}.
\]
Consequently \(\cM_3^{\trop}\) has exactly ten connected components, indexed by the words beginning with \(+1\); each is contractible and is the union of one generic cell with the faces obtained from it by same-sign collisions. Writing \(a=\min\bigl(\supp(D_+)\cup\supp(D_-)\bigr)\), the normalization \(a=0\) defines a global slice for source translation, so \(\cM_3^{\trop}\cong\R\times\cM_{3,\mathrm{red}}^{\trop}\), and \(\cM_{3,\mathrm{red}}^{\trop}\) also has ten contractible components.
\end{prop}

\begin{proof}
List the six points of \(D_+\) and \(D_-\), with multiplicity, as \(p_1\le\cdots\le p_6\), and label each by \(+1\) or \(-1\) according to the divisor it comes from. Because the supports are disjoint, equal points carry equal labels, so the labels form a well-defined word, namely \(w\); conversely, \((p_1,\ldots,p_6)\) and \(w\) determine \((D_+,D_-)\). The ordered coordinates depend continuously on the pair, and the word can change only when a zero and a pole coincide, which does not happen in \(U_3\). Hence each \(U_w\) is open and closed in \(U_3\). In the coordinates \(p_1\) and \(p_{i+1}-p_i\), the set \(U_w\) is defined by \(p_{i+1}-p_i>0\) for the \(r(w)-1\) indices with \(w_i\ne w_{i+1}\) and \(p_{i+1}-p_i\ge0\) for the other \(6-r(w)\), which gives the stated convex, hence contractible, set. Target reflection maps \(U_w\) to \(U_{-w}\) with \(w\ne-w\), so the components of \(\cM_3^{\trop}=U_3/\langle\iota\rangle\) are the \(U_w\) with \(w_1=+1\), of which there are \(\binom52=10\). Source translation changes only \(p_1=a\); translating every divisor point by \(-a\) gives a continuous normalization to the slice \(a=0\), proving the stated product decomposition.
\end{proof}

A cell of \(U_w\) is determined by which of the \(6-r(w)\) same-sign gaps are positive, and a cell with \(j\) such positive gaps has \(r(w)+j\) break points. The ten words beginning with \(+1\) have run numbers \(r=2,3,4,5,6\) with multiplicities \(1,2,4,2,1\), so
\[
\sum_{w_1=+1}q^{r(w)}(1+q)^{6-r(w)}=q^2+6q^3+16q^4+20q^5+10q^6,
\]
which recovers the cell counts of \cref{prop:cell-counts}.

\subsection{Extension to arbitrary degree}
\begin{rem}[General-degree divisor spaces]\label{rem:general-degree}
For any integer \(d\ge1\), let \(\cM_d^{\trop}\) denote the target-equivalence classes of divisor-degree-$d$ functions with finite zero and pole divisors, and set
\[
U_d=\bigl\{(D_+,D_-)\in\Sym^d(\R)\times\Sym^d(\R):
\supp(D_+)\cap\supp(D_-)=\varnothing\bigr\}.
\]
The argument of \cref{thm:global-model} gives \(\cM_d^{\trop}\cong U_d/\mathfrak S_2\). A word \(w\) now contains \(d\) letters of each sign. If \(r(w)\) is its number of runs, the ordered-point coordinates in the proof of \cref{prop:components} give
\[
U_w\cong\R\times(0,\infty)^{r(w)-1}\times[0,\infty)^{2d-r(w)}.
\]
Thus
\[
\dim\cM_d^{\trop}=2d,
\qquad
\#\pi_0(\cM_d^{\trop})=\frac12\binom{2d}{d}.
\]
Each component is contractible and contains exactly one generic cell. The component count follows because target reflection pairs \(w\) with \(-w\), without fixed words. Normalizing the leftmost divisor point to zero also gives
\(\cM_d^{\trop}\cong\R\times\cM_{d,\mathrm{red}}^{\trop}\), with reduced dimension \(2d-1\).

The composition argument in \cref{thm:degree3-enumeration} likewise counts the oriented jump sequences in every degree:
\[
N_d=\sum_{p=1}^{d}\sum_{q=1}^{d}
\binom{d-1}{p-1}\binom{d-1}{q-1}\binom{p+q}{p}.
\]
Indeed, the positive and negative magnitudes are compositions of \(d\) into \(p\) and \(q\) positive parts, and the two lists can be interleaved in \(\binom{p+q}{p}\) ways. The number of target-reflection cell types is \(N_d/2\); restricting the sum to \(p+q=k\) gives the corresponding counts with \(k\) break points. These are direct extensions of the divisor construction. The complete cubic orbit representatives and the reflection strata of \cref{sec-6} give the explicit degree-$3$ specialization.
\end{rem}

\section{Automorphism Groups and Symmetry Stratification}\label{sec-6}

For a tropical rational function \(F\), define
\[
\Aut(F)=\{(\alpha,\beta)\in\Aut(\TP^1)\times\Aut(\TP^1):
\beta\circ F=F\circ\alpha\}.
\]
This is a symmetry group of the function, not the equivalence relation defining \(\cM_3^{\trop}\). It is also not a symmetry of network weights in the sense of the ReLU literature, where a symmetry is a transformation of parameters leaving the computed function unchanged: every network has the permutation and positive-scaling symmetries, and some have additional hidden ones \cites{PhuongLampert20,GrigsbyLindseyRolnick23}; for shallow networks a complete classification of these parameter symmetries is given in \cite{Ramakrishnan26}. The group \(\Aut(F)\) is an affine change of input and output carrying the function to itself, invisible to the realization map. The point of contact with the parameter-space picture is elsewhere. Two units with the same threshold and the same orientation can be merged into one, and the same-sign collisions of \cref{sec-5} are the function-space shadow of this merging; units with the same threshold, the same orientation and opposite output weights cancel, which is the zero-pole collision that leaves \(\cM_3^{\trop}\). With opposite orientations they would instead leave an affine term, since \(\operatorname{ReLU}(x-\theta)-\operatorname{ReLU}(\theta-x)=x-\theta\). The choice among the twenty classes of minimal representations of a generic function (\cref{sec-2}) is a further parameter-space redundancy with no counterpart in the divisor.

\subsection{Rigidity}
\begin{prop}[Rigidity]\label{prop:rigidity-v3}
Let \(F\) be an exact degree-$3$ tropical rational function with finite nonempty divisor support. Then
\[
\Aut(F)\in\{\{1\},\Z/2\Z\}.
\]
\end{prop}

\begin{proof}
The image of a nonconstant continuous piecewise-linear function contains a nontrivial interval. Hence the target affine automorphism \(\beta\) is uniquely determined by \(\alpha\) from the identity \(\beta(F(x))=F(\alpha(x))\). Thus the source projection is injective.

Moreover, the breakpoint set of \(F\circ\alpha\) is \(\alpha^{-1}\) of the breakpoint set of \(F\), while post-composition by \(\beta\) does not change breakpoint locations. Therefore \(\alpha\) preserves the finite nonempty breakpoint set. An integral-affine automorphism of \(\R\) preserving a finite nonempty set \(S\) is either the identity or the uniquely possible reflection \(x\mapsto\min S+\max S-x\). Hence the stabilizer has order at most two.
\end{proof}

\subsection{The reflection criterion}
Let \(F\) be a divisor-degree-$3$ tropical rational function with finite zero and pole divisors, jump sequence \(\delta=(\delta_1,\ldots,\delta_k)\), and gap vector \(\ell=(\ell_1,\ldots,\ell_{k-1})\).

\begin{thm}[Symmetry criterion]\label{thm:symmetry-criterion}
A nontrivial automorphism exists if and only if
\begin{equation}\label{eq:antipal}
\delta_i=-\delta_{k+1-i}
\qquad (1\le i\le k)
\end{equation}
and
\begin{equation}\label{eq:gapsym}
\ell_i=\ell_{k-i}
\qquad (1\le i\le k-1).
\end{equation}
When these conditions hold, the nontrivial automorphism is represented by a source reflection and a target reflection.
\end{thm}

\begin{proof}
By \cref{prop:rigidity-v3}, a nontrivial source automorphism is a reflection \(\alpha(x)=c-x\). Suppose first that the target automorphism is orientation-preserving. Applying the symmetry twice shows that its translation part must vanish, so \(F(c-x)=F(x)\). This would make the jump sequence palindromic, \(\delta_i=\delta_{k+1-i}\). No degree-$3$ jump sequence is palindromic: if \(k\) is even, the positive divisor degree would be even; if \(k\) is odd, the central term together with doubled outer terms again makes the positive degree even. Hence the target automorphism must also be a reflection.

Thus the symmetry equation has the form
\[
F(c-x)=-F(x)+b.
\]
Differentiating on linearity intervals shows that the slope sequence is palindromic, equivalently the jump sequence is anti-palindromic as in \cref{eq:antipal}. The source reflection preserves the ordered breakpoint set exactly when the gaps satisfy \cref{eq:gapsym}. Conversely, if both conditions hold, take \(c=\theta_1+\theta_k\). The function \(F(c-x)+F(x)\) has zero slope on every linearity interval; continuity across the break points then makes it a single constant \(b\). This gives the required simultaneous reflection.
\end{proof}

\begin{cor}\label{cor:generic-symmetry}
On a generic six-breakpoint cell, nontrivial symmetry occurs exactly when the jump word is anti-palindromic and
\[
\ell_1=\ell_5,
\qquad
\ell_2=\ell_4.
\]
Thus the symmetric locus has codimension \(2\) and dimension \(4\) in each anti-palindromic six-dimensional maximal cell, and is empty in the other maximal cells.
\end{cor}

The anti-palindromic generic source-unoriented types are:
\begin{align*}
&(1,1,1,-1,-1,-1),\\
&(1,1,-1,1,-1,-1),\\
&(1,-1,1,-1,1,-1),\\
&(1,-1,-1,1,1,-1).
\end{align*}
Equivalently, among the ten target-reflection maximal types, four support a nonempty reflection-symmetry locus.

More generally, anti-palindromic degree-$3$ jump sequences occur only for even \(k\). There is one target type for \(k=2\), four for \(k=4\), and four for \(k=6\). The metric symmetry locus in an anti-palindromic \(k\)-breakpoint cell is cut out by \(\lfloor(k-1)/2\rfloor\) independent gap equations.

\section{Connection to Tropical Hurwitz Theory}\label{sec-7}

The moduli space \(\cM_3^{\trop}\) is motivated in part by the same divisor and ramification combinatorics that appears in Hurwitz theory. There is, however, an important distinction between the scalar tropical rational functions studied here and the harmonic morphisms of metric graphs used to define tropical Hurwitz spaces. The purpose of this section is to make that comparison precise without identifying the two moduli problems.

\subsection{Tropical covers and Hurwitz numbers}
A \emph{finite harmonic morphism} \(\Phi:\Gamma\to T\) of degree \(d\) between metric graphs is integer-affine on edges and satisfies the balancing condition at every point: the local degree, obtained by summing expansion factors over the tangent directions mapping to a fixed target direction, is independent of that direction. A \emph{tropical admissible cover} with all ramification marked carries more structure \cite{CavalieriMarkwigRanganathan16}: the vertices of source and target carry genera and markings, and at every vertex \(v\) over a vertex \(w\) the local Riemann--Hurwitz condition
\[
2g(v)-2=d_v\bigl(2g(w)-2\bigr)+\sum_{e\ni v}(m_e-1)
\]
holds, where \(d_v\) is the local degree and the sum runs over the edges and legs at \(v\) with their expansion factors \(m_e\). Balancing alone does not imply this condition: the map between two tripods with genus-zero central vertices that stretches each leg by a factor \(3\) is harmonic of degree \(3\), but the right-hand side equals \(3(-2)+3\cdot2=0\neq-2\). Local Riemann--Hurwitz is necessary but not sufficient for algebraic realizability: the moduli space of \cite{CavalieriMarkwigRanganathan16} retains only combinatorial types whose local Hurwitz numbers are nonzero. In the characteristic-zero setting of that work, with all ramification marked and a trivalent, totally degenerate target, nonvanishing of these local Hurwitz numbers is equivalent to realizability as the skeleton of an admissible cover. Tropical Hurwitz numbers count tropical covers of \(\TP^1\), in the sense of \cite{CavalieriJohnsonMarkwig10}*{Def.~5.3}, with lattice multiplicities and automorphism factors; see also \cite{Katz11}.

A scalar tropical rational function \(F:\TP^1\to\mathbb R\cup\{\pm\infty\}\), by contrast, is an integral piecewise-linear function specified by its principal divisor. It is not required to satisfy the balancing condition when regarded as a map of a fixed source graph. For example, at \(x=1\) the function in \cref{exa:degree3} has adjacent slopes \(1\) and \(2\). On the bare bivalent source path the two target directions therefore have different local degree sums, so the map is not harmonic. A genuine tropical cover carrying analogous data requires the appropriate source/target graph structure and balancing data.

This also fixes the relevant classical count. For connected genus-zero degree-$3$ covers of \(\PP^1\) with four fixed labeled simple branch points, each cover weighted by \(1/|\Aut|\), the simple Hurwitz number is \(4\). Each simple branch point contributes a transposition in \(S_3\). There are \(24\) transitive ordered quadruples \((\tau_1,\tau_2,\tau_3,\tau_4)\) of transpositions with product \(1\), and simultaneous conjugation by \(S_3\) acts freely on them, hence \(24/6=4\). With the unmarked-fiber convention of \cite{CavalieriJohnsonMarkwig10}*{Def.~2.1}, this is the double Hurwitz number \(H^0_3\bigl((1,1,1),(1,1,1)\bigr)\): the fibers over \(0\) and \(\infty\) are unramified and their points are not individually labeled. By \cite{CavalieriJohnsonMarkwig10}*{Thm.~5.28}, it equals the degree of the tropical branch map. Marking these two fibers instead gives \(3!\,3!\cdot4=144\), as in the convention adopted in \cite{CavalieriJohnsonMarkwig10}*{Section~6}.

\subsection{Breakpoint images versus divisor locations}
We work on the source-reduced space \(\cM_{3,\mathrm{red}}^{\trop}\) of \cref{sec:source-reduction}. Breakpoint images are values of the scalar function, not branch values of a harmonic cover. Fix a representative \(\delta=(\delta_1,\ldots,\delta_k)\) of a target type and write \(s_i=\delta_1+\cdots+\delta_i\). For ordered break points \(\theta_i\) and gaps \(\ell_i=\theta_{i+1}-\theta_i\), the relation
\[
F(\theta_{i+1})-F(\theta_i)=s_i\ell_i
\]
holds on every linearity interval. A degree-$3$ function need not be monotone, and its internal slopes \(s_1,\ldots,s_{k-1}\) may vanish. On a fixed reduced cell the map \((\ell_1,\ldots,\ell_{k-1})\mapsto(s_1\ell_1,\ldots,s_{k-1}\ell_{k-1})\) has rank equal to the number of nonzero internal slopes; negative slopes only replace the positive orthant by an orthant of another sign. So the breakpoint-image differences are coordinates on the reduced cell exactly when every internal slope is nonzero, and each vanishing internal slope makes the corresponding gap invisible. Among the ten generic target types, the two with slope sequences \((0,1,2,3,2,1,0)\) and \((0,1,2,1,2,1,0)\) give coordinates, four have one vanishing internal slope, and four have two. The divisor coordinates, which work uniformly on all cells, are therefore the natural global coordinates for the degree-$3$ theory.

\subsection{What survives from the Hurwitz analogy}
The comparison remains useful. Both theories organize discrete data into polyhedral cells; both have natural location coordinates; and both acquire special strata when marked points collide. In the present scalar theory, exact divisor degree fixes three units of positive and three units of negative divisor mass. In genuine Hurwitz theory, local expansion factors and harmonic balancing are part of the cover data and Riemann--Hurwitz is expressed through the canonical divisors of source and target graphs.

Accordingly, the correct conclusion is a comparison rather than an equality of moduli spaces. The divisor model gives an explicit scalar parameter space for cubic tropical rational functions and ReLU breakpoints, while tropical Hurwitz spaces give the balanced-cover refinement appropriate for enumerative geometry. Constructing a natural morphism from a suitably enriched divisor model to the standard degree-$3$ tropical Hurwitz space remains an interesting problem.

\section{Compactification}\label{sec-8}

\subsection{The intrinsic divisor-pair compactification}
The global description of \cref{thm:global-model} suggests a compactification without introducing coordinates chart by chart.

\begin{defn}\label{def:compactification}
Define
\[
\overline{\cM}_3^{\trop}
:=
\bigl(\Sym^3(\TP^1)\times\Sym^3(\TP^1)\bigr)/\mathfrak S_2,
\]
where \(\mathfrak S_2\) exchanges the two degree-$3$ divisors.
\end{defn}

Ordering the points of a divisor identifies \(\Sym^3(\TP^1)\) with the compact ordered-position space
\[
\Delta=\bigl\{(x_1,x_2,x_3)\in[-\infty,+\infty]^3:\ x_1\le x_2\le x_3\bigr\},
\]
which is homeomorphic to a closed \(3\)-simplex: an order-preserving homeomorphism \(h:\TP^1\to[0,1]\) identifies \(\Delta\) with
\(\{0\le u_1\le u_2\le u_3\le1\}\). This is a topological identification, not an assertion that the extended-real coordinates are ordinary affine simplex coordinates. Hence

\[
\overline{\cM}_3^{\trop}\cong(\Delta\times\Delta)/\mathfrak S_2,
\]
with the ordered positions \((x_1,x_2,x_3)\) of the zeros and \((y_1,y_2,y_3)\) of the poles as coordinates on \(\Delta\times\Delta\). In these coordinates the cell of an oriented jump sequence \(\delta\) is the locally closed subset on which all six coordinates are finite, no \(x_i\) equals any \(y_j\), and the ties and the interlacing order of the \(x_i\) and \(y_j\) are those prescribed by \(\delta\). Since \(\TP^1\) is compact, \(\overline{\cM}_3^{\trop}\) is compact, and \(\cM_3^{\trop}=U_3/\mathfrak S_2\) is dense in it, since finiteness and disjointness of the two supports are open conditions whose complement has empty interior; so this is a compactification.

A point of \(\overline{\cM}_3^{\trop}\) records the uncancelled pair \(\{D_+,D_-\}\). Since the pair is unordered, it determines a scalar function only up to target equivalence: the tropical rational function with divisor \(D_+-D_-\) given by \cref{prop:tree-divisor} is replaced by its negative when the two divisors are exchanged. This target-equivalence class forgets every cancelled zero-pole pair, so distinct boundary points can have the same class; see \cref{sec:root-pole}.

The boundary has two natural pieces.

\begin{enumerate}[label=(\roman*)]
\item \textbf{Toric boundary.} One or more zeros or poles lie at \(\pm\infty\). A zero and a pole at the same end cancel in the scalar function, exactly as on the resultant boundary.
\item \textbf{Tropical resultant boundary.} The positive and negative divisors share a finite point (coincident valuations of a zero and a pole, not necessarily a common algebraic root). The scalar principal divisor then cancels at the common point and its exact degree drops.
\end{enumerate}

Same-sign collisions are not boundary phenomena in this intrinsic compactification: they remain points of \(\Sym^3(\R)\) with multiplicity and correspond to lower-breakpoint cells inside \(\cM_3^{\trop}\).

\subsection{Degree drop on the resultant boundary}
Suppose a zero of multiplicity \(a\) and a pole of multiplicity \(b\) collide at the same point. The common divisor contains \(\min(a,b)\) copies of that point, so cancellation lowers both the positive and negative divisor degrees by
\[
\min(a,b),
\]
in agreement with \cref{eq:degree-drop}.

Let \(\cM_{2,\mathrm{gen}}^{\trop}\) denote the set of degree-$2$ tropical rational functions whose divisor consists of four distinct points of \(\R\), modulo target translations and reflection. By \cref{prop:tree-divisor}, its points are unordered pairs of effective degree-$2$ divisors with four distinct points in their combined support. Fixing their interlacing gives a four-dimensional cell parametrized by the ordered divisor positions.

\begin{prop}[Generic resultant boundary]\label{prop:degree2-boundary}
Let \(B\subset\overline{\cM}_3^{\trop}\) be the set of pairs \(\{D_+,D_-\}\) such that all six points of \(D_+\) and \(D_-\) lie in \(\R\), the divisors \(D_+\) and \(D_-\) share exactly one point \(p\), which is simple in each, and the remaining four points are pairwise distinct. Then
\[
\{D_+,D_-\}\longmapsto\bigl([G],p\bigr),
\qquad
\divisor(G)=(D_+-[p])-(D_--[p]),
\]
is a bijection from \(B\) onto
\[
\bigl\{([G],p):\ [G]\in\cM_{2,\mathrm{gen}}^{\trop},\ p\in\R\setminus\supp\,\divisor(G)\bigr\}.
\]
In particular \(B\) has dimension \(4+1=5\), and near each of its points it is a product of a generic degree-$2$ cell with an interval.
\end{prop}

\begin{proof}
Write \(D_\pm=[p]+D_\pm'\). Because \(p\) is simple in each divisor and is the only common point, \(p\notin\supp(D_+')\cup\supp(D_-')\), and \(D_+'\) and \(D_-'\) are disjoint. Because the remaining four points are distinct and finite, \(G\) has divisor degree \(2\) with four distinct points of \(\R\), so \([G]\in\cM_{2,\mathrm{gen}}^{\trop}\). Exchanging \(D_+\) and \(D_-\) replaces \(G\) by \(-G\), which has the same target-equivalence class, and does not move \(p\); so the map is well defined on \(\mathfrak S_2\)-orbits. Conversely, given \(([G],p)\) with \(p\notin\supp\,\divisor(G)\), the pair \(\{D_+(G)+[p],\,D_-(G)+[p]\}\) lies in \(B\), does not depend on the representative \(G\), and is sent back to \(([G],p)\). The conditions \(p\notin\supp\,\divisor(G)\) and distinctness of the four points are open, which gives the local product structure.
\end{proof}

Thus a generic resultant-boundary point records a degree-$2$ function and a cancellation point outside its divisor support. The remaining points of the resultant boundary, where a point lies at \(\pm\infty\), where the common point is multiple, or where further points coincide, form strata of smaller dimension.

\subsection{Gluing of cells and the gap coordinates}\label{sec:gluing}
In the coordinates of \(\Delta\times\Delta\) the gluing of cells is automatic: two families converge to the same point of \(\overline{\cM}_3^{\trop}\) exactly when the limits of their ordered zero and pole positions agree, up to \(\mathfrak S_2\). Inside \(\cM_3^{\trop}\) every collision is of the same sign, and this reproduces the identifications of \cref{sec-5}. In the compactification, cells also meet along zero-pole collision faces. For \(0<\varepsilon<1\) the families
\[
\bigl\{[-2]+[-1]+[-\varepsilon],\ [\varepsilon]+[1]+[2]\bigr\}
\quad\text{and}\quad
\bigl\{[-2]+[-1]+[\varepsilon],\ [-\varepsilon]+[1]+[2]\bigr\}
\]
lie in the generic cells of the jump words \((1,1,1,-1,-1,-1)\) and \((1,1,-1,1,-1,-1)\), and both converge to \(\bigl\{[-2]+[-1]+[0],\ [0]+[1]+[2]\bigr\}\) as \(\varepsilon\to0\). Thus faces of different cells are identified whenever their limiting uncancelled divisor pairs agree modulo \(\mathfrak S_2\); recording only the scalar divisor after cancellation would lose the cancellation point.

The gap coordinates of \cref{sec-5} do not extend to coordinates on \(\overline{\cM}_3^{\trop}\). With \(a=\theta_1\) and \(\theta_j=a+\ell_1+\cdots+\ell_{j-1}\), the map \((a,\ell)\mapsto\{D_+,D_-\}\) extends continuously to \(\R\times[0,\infty]^{k-1}\to\overline{\cM}_3^{\trop}\), but it is not injective on the faces where some \(\ell_i=\infty\): there every \(\theta_j\) with \(j>i\) is sent to \(+\infty\), and the later gaps are forgotten. For instance, for all \(u,v>0\),
\[
\bigl\{[0]+[1]+[2],\ [L]+[L+u]+[L+u+v]\bigr\}\longrightarrow\bigl\{[0]+[1]+[2],\ 3[+\infty]\bigr\}
\qquad(L\to\infty).
\]

Nor does \(\overline{\cM}_3^{\trop}\) induce a compactification of the source-reduced space \(\cM_{3,\mathrm{red}}^{\trop}\). Source translation extends to \(\overline{\cM}_3^{\trop}\) but fixes every pair supported in \(\{\pm\infty\}\), and the quotient is not even \(T_1\): the orbit of \(\{3[0],3[1]\}\) accumulates at the fixed point \(\{3[+\infty],3[+\infty]\}\), so it is not closed. Relatedly, a reduced gap tending to \(\infty\) does not determine a limit: \(\{3[0],3[L]\}\) and its translate \(\{3[-L/2],3[L/2]\}\) have the same gap \(L\), but converge to \(\{3[0],3[+\infty]\}\) and \(\{3[-\infty],3[+\infty]\}\) respectively. Extending each reduced gap to \([0,\infty]\) therefore compactifies each cell of \(\cM_{3,\mathrm{red}}^{\trop}\) separately; these extended cubes form a different object from \(\overline{\cM}_3^{\trop}\) and are not used below.

\subsection{Intersection-theoretic questions}
The compactification provides a natural combinatorial setting in which to formulate intersection-theoretic questions. Boundary strata correspond to degenerations of the divisor pair, and intersections encode simultaneous degenerations. For example, imposing two same-sign collision conditions records successive mergers of zeros or poles, while intersecting a collision condition with a toric-boundary condition records a merger together with escape toward an end of \(\TP^1\). Resultant strata record zero-pole cancellation and therefore interact with lower-degree moduli spaces.

At the level of the underlying cell structure these are set-theoretic intersections. To promote them to tropical intersection products one must supply compatible lattice structures, balancing weights, and multiplicities, and one may also wish to incorporate stable degenerations of the source curve. We leave this refinement to future work; compare \cites{Mikhalkin07,GathmannKerberMarkwig09,CavalieriHampeMarkwigRanganathan16,Ochse13} for related tropical moduli constructions.

\section{Classical Cubics and Tropicalization}\label{sec-9}

\subsection{Root-pole tropicalization}\label{sec:root-pole}
Let \(K\) be an algebraically closed field with additive non-Archimedean valuation
\[
\nu:K^*\longrightarrow\Gamma\subset\R.
\]
We use max-plus convention, so the tropical value of a coefficient is \(-\nu(\cdot)\). Consider a cubic rational function with all zeros and poles in the algebraic torus,
\begin{equation}\label{eq:cubic-factor}
f(z)=c\,\frac{\prod_{i=1}^3(z-a_i)}{\prod_{j=1}^3(z-b_j)},
\qquad a_i,b_j\in K^*.
\end{equation}
Set
\[
\lambda_i=-\nu(a_i),
\qquad
\mu_j=-\nu(b_j).
\]
The associated one-variable tropical rational function is, up to an additive constant,
\begin{equation}\label{eq:root-trop}
F_f(x)=
\sum_{i=1}^3\max\{x,\lambda_i\}
-
\sum_{j=1}^3\max\{x,\mu_j\}.
\end{equation}
Its divisor is \(\sum_i[\lambda_i]-\sum_j[\mu_j]\), after cancellation of common terms. This is the one-dimensional form of the slope formula; compare \cite{BakerPayneRabinoff16}.

We distinguish this scalar tropicalization from the \emph{divisor-pair tropicalization}
\[
\trop_2(f):=\Bigl\{\sum_{i=1}^3[\lambda_i],\ \sum_{j=1}^3[\mu_j]\Bigr\}\in\overline{\cM}_3^{\trop},
\]
which records the zero and pole valuations with multiplicity and without cancellation. Since \(\trop_2(f)\) is an unordered pair, it determines the target-equivalence class \([F_f]\), not \(F_f\) itself: \(f\) and \(1/f\) have the same \(\trop_2\), while \(F_{1/f}=-F_f\). An ordered pair would determine \(F_f\) up to an additive constant. On the exact-degree-$3$ locus the correspondence between \(\trop_2(f)\) and \([F_f]\) is bijective; on the boundary \([F_f]\) forgets the cancelled common divisor. Let \(m_+(\xi)\) and \(m_-(\xi)\) be the numbers of \(i\) with \(\lambda_i=\xi\) and of \(j\) with \(\mu_j=\xi\). Then the positive and negative parts of \(\divisor(F_f)\) have degree
\[
3-\sum_{\xi\in\R}\min\{m_+(\xi),m_-(\xi)\}.
\]
Hence \(F_f\) has exact divisor degree \(3\) if and only if \(\nu(a_i)\ne\nu(b_j)\) for all \(i,j\); coincidences among the zero valuations alone, or among the pole valuations alone, are allowed. In that case \(\trop_2(f)\in\cM_3^{\trop}\), and it carries the same information as \([F_f]\). If some \(\lambda_i=\mu_j\), then \(\trop_2(f)\) lies on the tropical resultant boundary of \cref{sec-8} and still records the cancelled pairs, while \(F_f\) has smaller degree. For example, over \(\mathbb C((t^{\R}))\), where nonzero complex constants have valuation \(0\), the cubic
\[
f(z)=\frac{(z-1)(z-2)(z-3)}{(z-4)(z-5)(z-6)}
\]
has \(F_f=0\), whereas \(\trop_2(f)=\{3[0],3[0]\}\).

\subsection{The target normalizer}
Let \(T\subset\PGL_2\) be the diagonal torus. Its normalizer is
\[
N(T)\cong K^*\rtimes\Z/2\Z,
\]
acting on the target coordinate by
\[
y\longmapsto c y,
\qquad
y\longmapsto c/y.
\]
Under tropicalization, multiplication by \(c\) adds the constant \(-\nu(c)\), while inversion sends \(F\) to \(-F\) up to an additive constant. Thus the target action of \(N(T)\) tropicalizes to translations by \(\Gamma\) and reflections; when \(\Gamma=\R\) these are exactly the target translations and reflections used in the definition of \(\cM_3^{\trop}\).

Let \(\Rat_{3,\mathrm{sep}}^{\times}(K)\) denote the subset, open in the valuation topology, of cubic rational functions of the form \cref{eq:cubic-factor} for which
\[
\nu(a_i)\ne\nu(b_j)
\qquad\text{for all }i,j.
\]

For the value group \(\Gamma=\nu(K^*)\), define the \(\Gamma\)-rational locus by
\[
\cM_3^{\trop}(\Gamma)
:=\bigl\{[F]\in\cM_3^{\trop}:
\supp(\divisor(F))\subseteq\Gamma\bigr\}.
\]
Since source translation is not quotiented out, this requires the divisor positions themselves, not merely their differences, to belong to \(\Gamma\).

\begin{thm}[Exact image of root--pole tropicalization]\label{thm:surjective-trop}
For \(K\) as in \cref{sec:root-pole}, root--pole tropicalization induces a surjection
\[
\Rat_{3,\mathrm{sep}}^{\times}(K)/N(T)_{\mathrm{target}}
\longrightarrow
\cM_3^{\trop}(\Gamma).
\]
Thus its image in \(\cM_3^{\trop}\) consists exactly of the \(\Gamma\)-rational points. In particular, it is surjective onto \(\cM_3^{\trop}\) when \(\Gamma=\R\), and its image is dense in every polyhedral cell when \(\Gamma\) is dense in \(\R\).
\end{thm}

\begin{proof}
Target multiplication leaves the zero and pole divisors unchanged, and target inversion exchanges them, so \(\trop_2\) descends to the quotient. Every image has divisor support in \(\Gamma\), since its support positions are negative valuations of elements of \(K^*\).

Conversely, let \([F]\in\cM_3^{\trop}(\Gamma)\) and choose a representative \(F\). Write
\[
D_+(F)=\sum_{i=1}^3[\lambda_i],
\qquad
D_-(F)=\sum_{j=1}^3[\mu_j],
\]
counting multiplicities, with disjoint support in \(\Gamma\). Choose \(a_i,b_j\in K^*\) with \(-\nu(a_i)=\lambda_i\) and \(-\nu(b_j)=\mu_j\), and take the rational function \cref{eq:cubic-factor} with \(c=1\). Disjointness gives \(\nu(a_i)\ne\nu(b_j)\) for every \(i,j\), hence \(a_i\ne b_j\). Thus there is no algebraic cancellation, and the function belongs to \(\Rat_{3,\mathrm{sep}}^{\times}(K)\). Its tropicalization has the same principal divisor as \(F\), so by \cref{prop:tree-divisor} it differs from \(F\) by an additive constant and represents the same point of \(\cM_3^{\trop}\).

If \(\Gamma\) is dense, approximate each distinct divisor-support position by an element of \(\Gamma\), keeping its multiplicity and preserving the strict order of the distinct positions. The resulting \(\Gamma\)-rational points lie in the same cell and converge to the prescribed point, proving cellwise density.
\end{proof}

For example, the field of Hahn series \(\mathbb C((t^{\R}))\) is algebraically closed with value group \(\R\), giving the full moduli space as image. The field of Puiseux series over \(\mathbb C\) is algebraically closed with value group \(\mathbb Q\), giving exactly the \(\mathbb Q\)-rational locus, which is dense in every cell. The theorem explains why divisor degree is the appropriate notion for the present paper: the entire six-dimensional tropical moduli space is realized by genuine cubic rational functions on the torus.

\subsection{Source equivariance}\label{sec:source-equivariance}
The same normalizer acts on the source. The transformations
\[
z\longmapsto az,
\qquad
z\longmapsto a/z
\]
act on root and pole valuations by translation and reflection of the tropical source. Therefore the map in \cref{thm:surjective-trop} is equivariant for source reparametrizations lying in \(N(T)\). It is not equivariant for arbitrary Möbius transformations: a general element of \(\PGL_2(K)\) involves addition, and valuation does not turn algebraic addition into tropical maximum in the presence of cancellation.

\subsection{Comparison with the classical conjugacy moduli space}
From here to the end of this section the base field has characteristic zero, where the invariant theory of \cite{2024-04} applies. The parameter space \(\Rat_3\) of cubic rational functions has dimension \(7\). Quotienting the open torus locus by the one-dimensional continuous part of the target normalizer gives a six-dimensional quotient, matching
\[
\dim\cM_3^{\trop}=6.
\]
By contrast, the conjugacy moduli space \(\cM_3^1\) of \cite{2024-04} is the quotient by the three-dimensional \(\PGL_2\)-conjugation action and has dimension \(4\). The difference comes from the two equivalence relations, post-composition by the target torus on one side and conjugation by all of \(\PGL_2\) on the other, and not from a feature of tropicalization. The closest tropical counterpart of dividing also by a source torus is the reduced space \(\cM_{3,\mathrm{red}}^{\trop}\) of dimension \(5\), which matches the quotient of \(\Rat_3\) by the target and source tori, of dimension \(7-2=5\); conjugation by elements of \(\PGL_2\) outside \(N(T)\) has no affine counterpart on valuations (\cref{sec:source-equivariance}).

Thus the present space should not be identified with \(\Trop(\cM_3^1)\) without further work. Rather, both spaces arise from the same classical family with different equivalence relations. The invariant-theoretic strata of \cite{2024-04} therefore remain a natural source of comparison, but a direct tropicalization theorem for those invariant equations would require a common embedding and a careful analysis of the full conjugation action.

\subsection{Invariant-theoretic structure and its tropical shadow}
The classical description in \cite{2024-04} maps the cubic conjugacy moduli space to a weighted projective space using explicit invariants: the ring of \(\SL_2\)-invariants of \(V_4\oplus V_2\) is generated by six invariants of degrees \(2,2,3,3,4,6\). This ring has Krull dimension \(8-3=5\), so the six generators satisfy a single relation, and the four-dimensional space \(\cM_3^1\) maps into a hypersurface of \(\mathbb P_{(2,2,3,3,4,6)}\). Its loci of maps with nontrivial automorphisms are cut out by algebraic relations among those invariants. Our divisor-degree model also has a sharply defined symmetry locus: on a generic anti-palindromic cell the nontrivial reflection occurs exactly under the two linear equations
\[
\ell_1=\ell_5,\qquad \ell_2=\ell_4.
\]
This parallel is one of the motivations for the construction.

What can presently be asserted is the structural analogy. To prove that these linear equations are literally a tropicalization of the classical invariant-theoretic conditions, one would have to choose a common embedding, tropicalize the classical invariant ideal, determine its fan structure and multiplicities, and compare the resulting cones with the cells of \(\cM_3^{\trop}\). That calculation is not carried out here. Accordingly, the symmetry stratification of \cref{sec-6} should be regarded as a tropical analogue of the invariant-theoretic stratification rather than as its already-proved tropicalization.

\subsection{Invariant valuations and failure of descent}\label{sec:weighted-perspective}
The classical invariant theory of \cites{2024-04,ShaskaDegreeFour26} suggests an ambient weighted-projective picture. Let \(I_1,\ldots,I_r\) be homogeneous invariants of positive weights \(w_1,\ldots,w_r\). In the max-plus convention put \(u_i=-\nu(I_i)\), with \(\nu(0)=+\infty\). Whenever not all the \(I_i\) vanish, these coordinates define a point of
\[
\TP(w_1,\ldots,w_r)_{\mathbb T}=(\mathbb T^r\setminus\{-\infty\}^r)/\sim,
\qquad
u\sim u+\lambda(w_1,\ldots,w_r),\quad\lambda\in\R.
\]
On the classical side, tropicalization is the closure of the negative-valuation image over a valued field with value group \(\R\); it gives a polyhedral subset on each coordinate stratum. Realizing such coordinates as functions of the divisor data would first require their descent through root-pole tropicalization. The following proposition gives an explicit obstruction.

For a cubic \(f(z)=P(z,1)/Q(z,1)\), with \(P,Q\) homogeneous cubic forms, let
\[
H_f(X,Y)=YP(X,Y)-XQ(X,Y)
\]
be the fixed-point quartic. For
\(H=AX^4+BX^3Y+CX^2Y^2+DXY^3+EY^4\), use the normalization
\begin{equation}\label{eq:quartic-invariants}
\begin{aligned}
I(H)&=12AE-3BD+C^2,\\
J(H)&=72ACE+9BCD-27AD^2-27B^2E-2C^3.
\end{aligned}
\end{equation}
These are \(I=6\xi_1\) and \(J=72\xi_2\) in \cite{2024-04}*{Section~3.3}. The ratio \(J^2/I^3\) is unchanged by scaling the quartic or changing its variables by an element of \(\GL_2\), and is therefore a rational conjugation invariant on \(I\ne0\). On \(IJ\ne0\), its valuation is finite and equals
\[
\nu(J^2/I^3)=3u_I-2u_J,
\qquad u_I=-\nu(I),\quad u_J=-\nu(J),
\]
which is independent of weighted-projective scaling. The ratio is undefined where \(I=J=0\).

\begin{prop}[Failure of descent of a conjugacy invariant]\label{prop:invariant-nondescent}
Let \(K=\mathbb C((t^{\R}))\), with \(\nu(t)=1\) and every nonzero complex constant of valuation zero. On the locus of \(\Rat^{\times}_{3,\mathrm{sep}}(K)\) where \(I(H_f)J(H_f)\ne0\), the function
\[
f\longmapsto\nu\!\left(\frac{J(H_f)^2}{I(H_f)^3}\right)
\]
is not determined by the root-pole tropicalization, even when the numerator and denominator are both monic. It is also not constant on target-scaling orbits, even when the associated tropical point lies in a generic cell. In particular, it does not descend to a function on \(\cM_3^{\trop}\).
\end{prop}

\begin{proof}
Consider
\[
f_0(z)=\frac{z^3+1}{(z-t)^3},
\qquad
f_1(z)=\frac{z^3-4z+1}{(z-t)^3}.
\]
Both numerators are monic with three nonzero complex roots, so all zero valuations are \(0\); all three pole valuations are \(1\). Thus both maps have exact algebraic degree \(3\), satisfy the valuation-separation condition, and have
\[
\trop_2(f_0)=\trop_2(f_1)=\{3[0],3[-1]\},
\qquad F_{f_0}=F_{f_1}.
\]
For the target-scaling assertion, take
\[
g(z)=\frac{(z-1)(z-t)(z-t^2)}{(z-t^3)(z-t^4)(z-t^5)}.
\]
The maps \(g\) and \(t^3g\) have the same zeros and poles, with valuations \(0,1,2\) and \(3,4,5\), respectively. Their common tropical point is generic, and they differ only by a target scaling.

Substitution in \cref{eq:quartic-invariants} gives the following lowest-order terms in \(t\). Here \(\operatorname{in}_t\) denotes the lowest nonzero term of the polynomial in \(t\).
\[
\begin{array}{c|cc|c}
h & \operatorname{in}_t I(H_h) & \operatorname{in}_t J(H_h)
  & \nu\bigl(J(H_h)^2/I(H_h)^3\bigr)\\ \hline
f_0 & -12 & -27 & 0\\
f_1 & 36t & 405 & -3\\
g & 1 & 2 & 0\\
t^3g & 13t^6 & 27t^8 & -2
\end{array}
\]
In particular all four maps lie in \(IJ\ne0\). The first two rows show failure of descent even with monic numerator and denominator. The last two rows show failure of invariance under target scaling on a generic tropical cell. Since \(\trop_2\) identifies the two maps in each pair, no function on \(\cM_3^{\trop}\) can give these invariant valuations for every lift.
\end{proof}

The monic pair exhibits loss of residue-field data, not an unfixed target scalar: the constant-order term of \(I\) cancels for \(f_1\), although its zero and pole valuations agree with those of \(f_0\). This example lies in a non-generic cell. The second example gives a generic-locus obstruction of a different, expected kind. Target post-composition is not conjugation: replacing \(f\) by \(cf\) replaces \(H_f\) by \(cYP-XQ\), generally not a scalar multiple of \(H_f\). There is therefore no reason for this conjugacy invariant to be constant on target-scaling orbits. The computation shows that even its valuation changes.

An invariant-coordinate construction on the divisor moduli space therefore requires additional restrictions, enriched tropical data, or a specified choice of lift. Passing to the full \(\PGL_2\)-conjugation quotient requires a further comparison; the surjectivity of \cref{thm:surjective-trop} does not resolve either obstruction.

\subsection{Coefficientwise convention}\label{sec:coefficientwise}
For completeness, if
\[
p(z)=\sum_{i=0}^3 a_i z^i,
\]
then in the max-plus convention associated to the additive valuation \(\nu\) we write
\[
\trop(p)(x)=\max_i\{-\nu(a_i)+ix\}.
\]
The minus sign is essential. In one variable the Newton polygon recovers the valuations of the roots, so the factorized expression \cref{eq:root-trop} is equivalent to the usual tropical-polynomial description up to an additive constant. Writing the construction through the roots and poles has the advantage that the divisor degree is visible immediately and that the resultant boundary is transparent.

\section{Conclusions}

We have constructed a six-dimensional tropical moduli space of divisor-degree-$3$ rational functions on the tropical projective line. The degree is defined intrinsically from the principal divisor: three zeros and three poles, counted with multiplicity. This restores direct compatibility with algebraic cubic rational functions and removes the need for an auxiliary asymptotic-degree convention.

The combinatorial classification is complete. There are \(106\) oriented jump sequences, \(53\) target-reflection types, and \(31\) source-unoriented types. The generic locus has six simple break points and seven source-unoriented combinatorial types. Each fixed type gives a polyhedral cell \(\R_{>0}^{k-1}\times\R\), and the ten target-reflection generic types form the six-dimensional maximal cells. Same-sign collisions remain inside the moduli space, while opposite-sign collisions cross the tropical resultant and lower the exact divisor degree; accordingly the space has ten contractible connected components, one for each generic type.

The affine symmetry theory is equally explicit. Every exact degree-$3$ function with finite support has automorphism group either trivial or \(\Z/2\Z\). A nontrivial symmetry occurs precisely when the jump sequence is anti-palindromic and the metric breakpoint configuration is symmetric. On a generic cell this is the codimension-$2$ condition
\[
\ell_1=\ell_5,
\qquad
\ell_2=\ell_4.
\]
Exactly four of the seven generic source-unoriented types admit such a locus.

The compactification
\[
\overline{\cM}_3^{\trop}
=
\bigl(\Sym^3(\TP^1)\times\Sym^3(\TP^1)\bigr)/\mathfrak S_2
\]
provides a natural global setting for toric and resultant degenerations. It is homeomorphic to the quotient of a product of two closed simplices. Its points record uncancelled divisor pairs, and its cells meet along zero-pole as well as same-sign collision faces. A generic resultant-boundary point records a degree-$2$ function together with a marked cancellation point. The source-reduced gap coordinates do not extend to it.

The ReLU interpretation becomes especially natural in this formulation. Generic degree-$3$ functions are one-dimensional shallow networks with six distinct thresholds, three positive unit jumps and three negative unit jumps, while multiplicity strata correspond to coincident thresholds or larger integral output coefficients. The hinge form is one of twenty classes of minimal representations of a generic such function. This places the moduli problem directly inside the established tropical-rational/ReLU correspondence.

Finally, the root-pole tropicalization theorem shows that every point of \(\cM_3^{\trop}\) is realized by a genuine cubic rational function over a field with value group \(\R\), after quotienting the target by the torus normalizer. The remaining gap to the four-dimensional conjugacy moduli space \(\cM_3^1\) is therefore not a failure of degree. It comes from the different quotients and from the information that root-pole tropicalization forgets: classical invariant valuations need not descend to \(\cM_3^{\trop}\), as \cref{prop:invariant-nondescent} proves for \(J^2/I^3\). Understanding how the invariant-theoretic equations of \cite{2024-04} interact with the torus-normalizer tropicalization, and whether a suitable further quotient or tropical modification recovers a meaningful tropical shadow of the classical conjugacy moduli space, is a natural direction for future work.

Several problems remain open. One is to enrich the scalar divisor model to a balanced source graph and construct a direct comparison morphism to a standard tropical Hurwitz space. Another is to determine the weighted balanced structure, if any, naturally carried by the compactification and to develop its intersection theory. A third is to find restrictions or enrichments of the tropical data under which the classical cubic invariants descend, and so to make precise the relation between the linear symmetry strata found here and the algebraic automorphism strata of \cite{2024-04}. The global model, component count, and oriented enumeration already extend to every degree by \cref{rem:general-degree}. Further work may develop explicit source-reflection orbit classifications, their symmetry strata, and the resulting ReLU function spaces in higher degree; applications of the moduli space to symmetry detection and functional equivalence for ReLU networks will be developed in the follow-up paper \cite{sh-121}.

\bibliography{sh-117}

\end{document}

%% file: macro.tex
\usepackage[T1]{fontenc}
\usepackage{amsmath,amsfonts,amssymb,amsaddr}
\usepackage{amsthm}
\usepackage{enumitem}
\usepackage{array}
\usepackage{hyperref}
\usepackage[msc-links]{amsrefs}
\usepackage[capitalise]{cleveref}

\hypersetup{colorlinks=true,urlcolor=blue,linkcolor=blue,citecolor=red}
\crefname{thm}{Thm.}{Thms.}
\Crefname{thm}{Thm.}{Thms.}
\crefname{prop}{Prop.}{Props.}
\Crefname{prop}{Prop.}{Props.}
\crefname{lem}{Lem.}{Lems.}
\Crefname{lem}{Lem.}{Lems.}
\crefname{cor}{Cor.}{Cors.}
\Crefname{cor}{Cor.}{Cors.}
\crefname{defn}{Def.}{Defs.}
\Crefname{defn}{Def.}{Defs.}
\crefname{rem}{Rem.}{Rems.}
\Crefname{rem}{Rem.}{Rems.}
\crefname{exa}{Ex.}{Exs.}
\Crefname{exa}{Ex.}{Exs.}
\crefname{prob}{Prob.}{Probs.}
\Crefname{prob}{Prob.}{Probs.}
\crefname{alg}{Alg.}{Algs.}
\Crefname{alg}{Alg.}{Algs.}
\crefname{section}{Section}{Sections}
\Crefname{section}{Section}{Sections}
\crefname{subsection}{Section}{Sections}
\Crefname{subsection}{Section}{Sections}
\crefname{subsubsection}{Section}{Sections}
\Crefname{subsubsection}{Section}{Sections}
\crefname{figure}{Fig.}{Figs.}
\Crefname{figure}{Fig.}{Figs.}
\crefname{table}{Table}{Tables}
\Crefname{table}{Table}{Tables}
\crefname{enumi}{}{}
\Crefname{enumi}{}{}
\crefformat{equation}{(#2#1#3)}
\Crefformat{equation}{(#2#1#3)}
\crefrangeformat{equation}{(#3#1#4)--(#5#2#6)}
\Crefrangeformat{equation}{(#3#1#4)--(#5#2#6)}
\crefmultiformat{equation}{(#2#1#3)}{ and~(#2#1#3)}{, (#2#1#3)}{ and~(#2#1#3)}
\Crefmultiformat{equation}{(#2#1#3)}{ and~(#2#1#3)}{, (#2#1#3)}{ and~(#2#1#3)}

\newtheorem*{thm*}{Theorem}
\newtheorem{thm}{Theorem}
\newtheorem{prop}{Proposition}

\newtheorem{cor}{Corollary}
\newtheorem{exa}{Example}

\theoremstyle{definition}
\newtheorem{defn}{Definition}

\theoremstyle{remark}
\newtheorem{rem}{Remark}

\newcommand\Z{\mathbb Z}

\newcommand\PP{\mathbb P}

\newcommand{\TP}{\mathbb{T}\PP}

\newcommand{\R}{\mathbb{R}}

\newcommand{\Aut}{\mathrm{Aut}}
\newcommand{\PGL}{\mathrm{PGL}}
\newcommand{\SL}{\mathrm{SL}}
\newcommand{\GL}{\mathrm{GL}}
\newcommand{\cM}{\mathcal{M}}

\newcommand{\trop}{\mathrm{trop}}
\newcommand{\Trop}{\mathrm{Trop}}
\newcommand{\divisor}{\mathrm{div}}
\newcommand{\Sym}{\mathrm{Sym}}
\newcommand{\Rat}{\mathrm{Rat}}
\newcommand{\ord}{\mathrm{ord}}
\newcommand{\supp}{\mathrm{supp}}

%% file: sh-117.bib
@unpublished{2024-04,
  author = {Badr, E. and Shaska, E. and Shaska, T.},
  title  = {Rational functions on the projective line from a computational viewpoint},
  year   = {2025},
  note   = {\href{https://arxiv.org/abs/2503.10835v2}{arXiv:2503.10835v2}, revised 2026},
}

@unpublished{ShaskaDegreeFour26,
  author = {Shaska, T.},
  title  = {The invariant ring of degree-four rational maps on the projective line},
  year   = {2026},
  note   = {\href{https://arxiv.org/abs/2608.23321}{arXiv:2608.23321}},
}

@unpublished{sh-121,
  author = {Shaska, T. and Siadat, M.-R.},
  title  = {{ReLU} networks via tropical geometry},
  year   = {2026},
  note   = {Preprint},
}

@book{MaclaganSturmfels15,
  author    = {Maclagan, D. and Sturmfels, B.},
  title     = {Introduction to tropical geometry},
  series    = {Graduate Studies in Mathematics},
  volume    = {161},
  publisher = {American Mathematical Society},
  address   = {Providence, RI},
  year      = {2015},
}

@inproceedings{Mikhalkin07,
  author    = {Mikhalkin, G.},
  title     = {Moduli spaces of rational tropical curves},
  booktitle = {{Proceedings of G\"okova Geometry-Topology Conference 2006}},
  pages     = {39--51},
  year      = {2007},
  note      = {\href{https://arxiv.org/abs/0704.0839}{arXiv:0704.0839}},
}

@article{GathmannKerberMarkwig09,
  author  = {Gathmann, A. and Kerber, M. and Markwig, H.},
  title   = {Tropical fans and the moduli spaces of tropical curves},
  journal = {Compos. Math.},
  volume  = {145},
  number  = {1},
  pages   = {173--195},
  year    = {2009},
}

@article{CavalieriJohnsonMarkwig10,
  author  = {Cavalieri, R. and Johnson, P. and Markwig, H.},
  title   = {Tropical {H}urwitz numbers},
  journal = {J. Algebraic Combin.},
  volume  = {32},
  number  = {2},
  pages   = {241--265},
  year    = {2010},
}

@phdthesis{Katz11,
  author       = {Katz, B. P.},
  title        = {Tropical {H}urwitz spaces},
  organization = {University of Texas at Austin},
  year         = {2011},
  note         = {\href{https://arxiv.org/abs/1211.2369}{arXiv:1211.2369}},
}

@article{CavalieriMarkwigRanganathan16,
  author  = {Cavalieri, R. and Markwig, H. and Ranganathan, D.},
  title   = {Tropicalizing the space of admissible covers},
  journal = {Math. Ann.},
  volume  = {364},
  pages   = {1275--1313},
  year    = {2016},
}

@article{CavalieriHampeMarkwigRanganathan16,
  author  = {Cavalieri, R. and Hampe, S. and Markwig, H. and Ranganathan, D.},
  title   = {Moduli spaces of rational weighted stable curves and tropical geometry},
  journal = {Forum Math. Sigma},
  volume  = {4},
  pages   = {Paper No. e9},
  year    = {2016},
}

@phdthesis{Ochse13,
  author       = {Ochse, D.},
  title        = {Moduli spaces of rational tropical stable maps into smooth tropical varieties},
  organization = {TU Kaiserslautern},
  year         = {2013},
}

@unpublished{JoynerKsirMelles10,
  author = {Joyner, D. and Ksir, A. and Melles, C. G.},
  title  = {Automorphism groups on tropical curves: some cohomology calculations},
  year   = {2010},
  note   = {\href{https://arxiv.org/abs/1006.4869}{arXiv:1006.4869}},
}

@unpublished{Song21,
  author = {Song, J.},
  title  = {Realizations of automorphism groups of metric graphs induced by rational maps},
  year   = {2021},
  note   = {\href{https://arxiv.org/abs/2103.00174}{arXiv:2103.00174}},
}

@article{BakerPayneRabinoff16,
  author  = {Baker, M. and Payne, S. and Rabinoff, J.},
  title   = {Nonarchimedean geometry, tropicalization, and metrics on curves},
  journal = {Algebr. Geom.},
  volume  = {3},
  number  = {1},
  pages   = {63--105},
  year    = {2016},
}

@article{TranWang24,
  author  = {Tran, N. M. and Wang, J.},
  title   = {Minimal representations of tropical rational functions},
  journal = {Algebr. Stat.},
  volume  = {15},
  number  = {1},
  pages   = {27--59},
  year    = {2024},
}

@unpublished{Sukenaga24,
  author = {Sukenaga, M.},
  title  = {Minimum volumes of tropical rational functions},
  year   = {2024},
  note   = {\href{https://arxiv.org/abs/2409.18348}{arXiv:2409.18348}},
}

@inproceedings{ZhangNaitzatLim18,
  author    = {Zhang, L. and Naitzat, G. and Lim, L.-H.},
  title     = {Tropical geometry of deep neural networks},
  booktitle = {{Proceedings of the 35th International Conference on Machine Learning ({ICML} 2018)}},
  series    = {Proceedings of Machine Learning Research},
  volume    = {80},
  pages     = {5824--5832},
  year      = {2018},
}

@unpublished{CharisopoulosMaragos18,
  author = {Charisopoulos, V. and Maragos, P.},
  title  = {A tropical approach to neural networks with piecewise linear activations},
  year   = {2018},
  note   = {\href{https://arxiv.org/abs/1805.08749}{arXiv:1805.08749}},
}

@article{DereichKassing22,
  author  = {Dereich, S. and Kassing, S.},
  title   = {On minimal representations of shallow {ReLU} networks},
  journal = {Neural Networks},
  volume  = {148},
  pages   = {121--128},
  year    = {2022},
  note    = {\href{https://arxiv.org/abs/2108.05643v1}{arXiv:2108.05643v1}},
}

@article{SahsEtAl22,
  author  = {Sahs, J. and Pyle, R. and Damaraju, A. and Ortega Caro, J. and Tavaslioglu, O. and Lu, A. and Anselmi, F. and Patel, A. B.},
  title   = {Shallow univariate {ReLU} networks as splines: initialization, loss surface, {H}essian, and gradient flow dynamics},
  journal = {Front. Artif. Intell.},
  volume  = {5},
  pages   = {Art. 889981},
  year    = {2022},
}

@inproceedings{PhamGarg24,
  author    = {Pham, T. A. and Garg, V.},
  title     = {What do graph neural networks learn? {I}nsights from tropical geometry},
  booktitle = {{Advances in Neural Information Processing Systems 37 ({NeurIPS} 2024)}},
  pages     = {10988--11020},
  year      = {2024},
}

@article{BrandenburgLohoMontufar24,
  author  = {Brandenburg, M.-C. and Loho, G. and Mont\'ufar, G.},
  title   = {The real tropical geometry of neural networks for binary classification},
  journal = {Trans. Mach. Learn. Res.},
  year    = {2024},
}

@article{BrandenburgGrilloHertrich25,
  author  = {Brandenburg, M.-C. and Grillo, M. and Hertrich, C.},
  title   = {Decomposition polyhedra of piecewise linear functions},
  journal = {Discrete Comput. Geom.},
  year    = {2026},
  note    = {\href{https://doi.org/10.1007/s00454-026-00875-1}{doi:10.1007/s00454-026-00875-1}},
}

@inproceedings{PhuongLampert20,
  author    = {Phuong, M. and Lampert, C. H.},
  title     = {Functional vs.\ parametric equivalence of {ReLU} networks},
  booktitle = {{International Conference on Learning Representations ({ICLR} 2020)}},
  year      = {2020},
}

@inproceedings{GrigsbyLindseyRolnick23,
  author    = {Grigsby, J. E. and Lindsey, K. and Rolnick, D.},
  title     = {Hidden symmetries of {ReLU} networks},
  booktitle = {{Proceedings of the 40th International Conference on Machine Learning ({ICML} 2023)}},
  series    = {Proceedings of Machine Learning Research},
  volume    = {202},
  pages     = {11734--11760},
  year      = {2023},
}

@unpublished{Ramakrishnan26,
  author = {Ramakrishnan, P.},
  title  = {A complete symmetry classification of shallow {ReLU} networks},
  year   = {2026},
  note   = {\href{https://arxiv.org/abs/2604.14037}{arXiv:2604.14037}},
}
